\documentclass[12pt,reqno]{amsart}
\usepackage[colorlinks=true, pdfstartview=FitV, linkcolor=blue, citecolor=blue, urlcolor=blue]{hyperref}

\usepackage{amscd, amssymb,latexsym,amsmath,  amsmath}
\usepackage[all]{xy}
\usepackage{anysize}
\usepackage[sc]{mathpazo} 
\usepackage{color}
\usepackage{hyperref}

\hypersetup{
    colorlinks=true,
    linkcolor=blue,
    filecolor=magenta,      
    urlcolor=cyan,
    pdftitle={Overleaf Example},
    pdfpagemode=FullScreen,
    }
   
\marginsize{2.5cm}{2.5cm}{2.5cm}{2.5cm}
\newtheorem{theorem}{Theorem}[section] 
\newtheorem{lemma}[theorem]{Lemma}

\theoremstyle{definition}

\newtheorem{example}{Example}

\theoremstyle{remark}
\newtheorem{remark}{Remark}

\newcommand{\pa}{\partial}

\newcommand{\Om}{\Omega}

\newcommand{\ga}{\gamma}

\newcommand{\la}{\lambda}

\newcommand{\diag}{\operatorname{diag}}

\newcommand{\Ric}{\operatorname{Ric}}
\renewcommand{\Re}{\operatorname{Re}}
\renewcommand{\Im}{\operatorname{Im}}

\begin{document}

\title[the KOBAYASHI--FUKS METRIC on strictly pseudoconvex domains]{On the Boundary Behaviour of Invariants and Curvatures of the Kobayashi--Fuks Metric in Strictly Pseudoconvex Domains}
\author{Anjali Bhatnagar}
\address{Indian Institute of Science Education and Research Pune, Pune~411008, India}
\email{anjali.bhatnagar@students.iiserpune.ac.in}

\begin{abstract}
The purpose of this article is to investigate the boundary behaviour of the Kobayashi--Fuks metric and several associated invariants on strictly pseudoconvex domains in the paradigm of scaling. This approach allows us
to examine more invariants, such as the canonical invariant, holomorphic sectional
curvature, and Ricci curvature of this metric, in a manner that extends and refines some existing analysis.
\end{abstract}

\date{}

\maketitle

  
\section{Introduction}
The Riemann mapping theorem is a landmark result in complex analysis, stating that any simply connected planar domain, except the complex plane $\mathbb{C}$, is holomorphically equivalent to the unit disc $\mathbb{D}$. Bergman introduced the Bergman representative coordinates---a tool in his program to generalize the Riemann mapping theorem in higher dimensions. The Kobayashi–Fuks metric is crucial in studying the Bergman representative coordinates. It was highlighted by Dinew in \cite{d11}, who initiated the study of this metric. Subsequently, in \cite{d13}, Dinew investigated the completeness of the Kobayashi–Fuks metric across a broad class of pseudoconvex domains. Borah-Kar \cite{bk22} examined the boundary behaviour of the Kobayashi–Fuks metric and some of its associated invariant objects for strictly pseudoconvex domains---using the localization of these invariants for pseudoconvex domains near local holomorphic peak points. It was achieved by expressing these invariants in terms of certain maximal domain functions, a technique inspired by Krantz-Yu \cite{kgy96}. Krantz-Yu draws inspiration from the classical approach of Bergman \cite{bg70}, which utilizes the minimum integral method and Bedford-Pincuk's scaling method \cite{bp91,p91}. Similar ideas are found in several other works (cf. \cite{k92, y95, ky96}). It is of natural interest to study the boundary behaviour of invariant objects without using their localization principle, as this simplifies the classical approach and offers a fresh perspective.

In this article, we study the boundary behaviour of the Kobayashi-Fuks metric and its associated invariants in strictly pseudoconvex domains, utilizing the scaling principle while bypassing the localization of these invariants. By doing so, we can explore more invariants, such as the canonical invariant, holomorphic sectional curvature, and the Ricci curvature of the Kobayashi-Fuks metric in a more natural manner as compared to \cite{bk22}. To set the stage, let \(\Omega \subset \mathbb{C}^n\) be a bounded domain. The \textit{Bergman space} $A^2(\Omega)$ is defined as the space of holomorphic functions in $\Omega$, which are square-integrable with respect to the Lebesgue volume measure on $\Omega$. It is a reproducing kernel Hilbert space whose reproducing kernel \(K_{\Omega}(z,w)\) is known as the \textit{Bergman kernel}. The Bergman kernel is uniquely characterized by the following: for each \(z \in \Omega\), \(K_{\Omega}(\cdot, z) \in A^2(\Omega)\); for all \(z, w \in \Omega\), \(K_{\Omega}(z,w) = \overline{K_{\Omega}(w,z)}\); and for each \(f \in A^2(\Omega)\),
\[
f(z) = \int_{\Omega} f(w) K_{\Omega}(z,w) \, dV\;\;\;\text{ for all }z\in\Omega.
\]
Moreover, in terms of any complete orthonormal basis \(\{\phi_k\}\) of \(A^2(\Omega)\), 
\begin{equation}\label{Berg-ker-onb}
    K_{\Omega}(z,w) = \sum_{k} \phi_k(z) \overline{\phi_{k}(w)},
\end{equation}
where the series converges uniformly on compact subsets of \(\Omega \times \Omega\). Consecutively, the function $\log K_\Omega(z,z)$ is a $C^\infty$-smooth, strictly plurisubharmonic function. Thus, it is a potential for a K\"ahler metric called the Bergman metric $ds^2_{b,\Omega}$ on $\Omega$, which is defined by
\[
ds^2_{b,\Omega}=\sum_{\alpha,\beta=1}^n g^{b,\Omega}_{\alpha\overline \beta} (z) \, dz_{\alpha}d\overline z_{\beta},
\]
where
\[
g^{b,\Omega}_{\alpha \overline \beta}(z)=\frac{\partial^2 \log K_{\Omega}(z,z)}{\partial z_{\alpha} \partial \overline z_{\beta}}.
\]
The Riemannian volume
element of $ds^2_{b,\Omega}$ is given by
\[
g_{b,\Omega}(z)=\det G_{b,\Omega}(z),\quad\text{where}\quad G_{b,\Omega}(z)=\begin{pmatrix}g^{b,\Omega}_{\alpha \overline \beta}(z)\end{pmatrix}_{n \times n}.
\]
The Ricci curvature of $ds^2_{b,\Omega}$ is given by
\begin{equation}\label{Ric_h-curv}
\text{Ric}_{b,\Omega}(z,X)=\frac{\sum_{\alpha, \beta=1}^n \Ric_{\alpha \overline \beta}^{b,\Omega}(z) X^{\alpha} \overline X^{\beta}}{\sum_{\alpha,\beta=1}^n g^{b,\Omega}_{\alpha\overline \beta}(z)X^\alpha \overline X^\beta}, \text{ where }\Ric_{\alpha \overline \beta}^{b,\Omega}(z)=  - \frac{\partial^2 \log g_{b,\Omega}(z)}{\partial z_{\alpha} \partial \overline z_{\beta}}.
\end{equation}
In \cite{Kob59}, Kobayashi demonstrated that the Ricci curvature of the Bergman metric is strictly bounded above by $n+1$ on bounded domains of $\mathbb C^n$, and thus the matrix
\[
G_{\tilde b,\Omega}(z)=\begin{pmatrix}g^{\tilde b,\Omega}_{\alpha \overline \beta}(z)\end{pmatrix}_{n \times n},
\]
where
\[
g^{\tilde b,\Omega}_{\alpha \overline \beta}(z)=(n+1)g^{b,\Omega}_{\alpha\overline \beta}(z)-\Ric^{b,\Omega}_{\alpha\overline \beta}(z) = \frac{\partial^2 \log \Big(K_\Omega(z,z)^{n+1}g_{b,\Omega}(z)\Big)}{\partial z_{\alpha} \partial \overline z_{\beta}}
\]
is positive definite; see \cite{Fuks66}. This induces the Kobayashi--Fuks metric on $\Omega$, as follows
\[ds^2_{\tilde b,\Omega}=\sum_{\alpha,\beta=1}^n g^{\tilde b,\Omega}_{\alpha \overline \beta}(z)\,dz_{\alpha}d\overline z_{\beta}.
\]
Clearly, $ds^2_{\tilde b,\Omega}$ is a K\"ahler metric, with the K\"ahler potential $\log \Big(K_\Omega(z,z)^{n+1}g_{b,\Omega}(z)\Big)$. Furthermore, if $F: \Omega_1 \to \Omega_2$ is a biholomorphism, then
\begin{equation}\label{tr-kf}
G_{\tilde b, \Omega_1}(z)=J_\mathbb C F(z)^t\, G_{\tilde b, \Omega_2}\big(F(z)\big) \overline {J_\mathbb C F(z)},
\end{equation}
 where $J_\mathbb C F$ is the complex Jacobian matrix of $F$. Hence, $ds^2_{\tilde b, \Omega}$ is an invariant metric. Here are some examples from \cite{bk22}.
 \begin{example}\label{ball}
Let $\mathbb{B}^n \subset \mathbb{C}^n$ be the unit ball. Then,
\[
ds^2_{\tilde b, \mathbb{B}^n}=(n+2)ds^2_{b, \mathbb{B}^n} = (n+1)(n+2)\sum_{\alpha,\beta=1}^n \left(\frac{\delta_{\alpha\overline\beta}}{1-\vert z \vert^2}+\frac{\overline z_{\alpha}z_{\beta}}{(1-\vert z \vert^2)^2}\right) dz_{\alpha} d\overline z_{\beta}.
\]
 \end{example}
 \begin{example}
     Let $\mathbb{D}^n \subset \mathbb{C}^n$ be the unit polydisc. Then,
     \[ds^2_{\tilde b, \mathbb{D}^n}=(n+2)ds^2_{b, \mathbb{D}^n} =2(n+2)\sum_{\alpha=1}^n\frac{1}{(1-|z_\alpha|^2)^2}|dz_\alpha |^2.\]
 \end{example}

To state our main results, we introduce some further notations. Let $h$ denote either $b$ or $\tilde b$. Then, we define
\[
G_{h,\Omega}(z)=\begin{pmatrix} g^{h,\Omega}_{\alpha \overline \beta}(z)\end{pmatrix}_{n \times n} \quad \text{and} \quad g_{h,\Omega}(z)=\det G_{h,\Omega}(z),
\]
which induces a canonical invariant
\[\beta_{h,\Omega}(z)=\frac{g_{h,\Omega}(z)}{K_\Omega(z)}.\]
The length of a tangent vector $X$ at $z \in \Omega$ in $ds^2_{h,\Omega}$ is denoted by $ds_{h,\Omega}(z,X)$, where
\[
ds^2_{h,\Omega}(z,X)=\sum_{\alpha,\beta=1}^n g^{h,\Omega}_{\alpha \overline \beta}(z)X^{\alpha} \overline X^{\beta}.
\]
The holomorphic sectional curvature of $ds^2_{h,\Omega}$ is given by
\begin{equation}\label{hsc-h}
R_{h,\Omega}(z,X)=\frac{\sum_{\alpha,\beta,\ga,\delta=1}^n R^{h,\Omega}_{\overline \alpha \beta \ga \overline \delta}(z)\overline X^{\alpha} X^{\beta} X^{\ga} \overline X^{\delta}}{ds^2_{h,\Omega}(z,X)^2},
\end{equation}
where
\begin{equation}
R^{h,\Omega}_{\overline \alpha \beta \ga \overline \delta}(z)=-\frac{\partial^2 g^{h,\Omega}_{\beta \overline \alpha}}{\partial z_{\ga} \partial \overline z_{\delta}}(z)+\sum_{\mu,\nu} g_{h,\Omega}^{\nu \overline \mu}(z)\frac{\partial g^{h,\Omega}_{\beta \overline \mu}}{\partial z_{\ga}}(z)\frac{\partial g^{h,\Omega}_{\nu \overline \alpha}}{\partial \overline z_{\delta}}(z).
\end{equation}
Here, $g_{h,\Omega}^{\nu \overline \mu}(z)$ being the $(\nu,\mu)$-th entry of $G_{h,\Omega}(z)^{-1}$. The Ricci curvature of $ds^2_{h,\Omega}$ is defined by \eqref{Ric_h-curv} with $b$ replaced by $h$.

To state the main results of this article, we recall the following notations. Let \( \delta_\Omega(z) \) denote the Euclidean distance from the point \( z \in \Omega \) to the boundary \( \partial \Omega \). For \( z \) sufficiently close to \( \partial \Omega \), let \( \pi(z) \in \partial \Omega \) be the nearest point to \( z \) such that \( \delta_\Omega(z) = |z - \pi(z)| \) and for a tangent vector \( X \in \mathbb{C}^n \) based at \( z \), we can decompose \( X \) as \( X = X_H(z) + X_N(z) \), where \( X_H(z) \) and \( X_N(z) \) represent the components along the tangential and normal directions, respectively at \( \pi(z) \). Let \( \mathcal{L}_{\partial \Omega} \) denote the Levi form of \( \partial \Omega \) with respect to some defining function of \( \Omega \).
\begin{theorem}\label{bdy-n}
    Let $\Omega \subset \mathbb{C}^n$ be a $C^2$-smoothly bounded strictly pseudoconvex domain, and let $p^0 \in \partial \Omega$. Then as $z\to p^0$, we have
\begin{align*}
&\text{\em (a)}& &\delta_\Omega(z)^{n+1} g_{\tilde b,\Omega}(z) \to \frac{(n+1)^n(n+2)^n}{2^{n+1}},&\\
&\text{\em (b)}& &\delta_{\Omega}(z)\, ds_{\tilde b,\Omega}(z,X_N(z))\to \frac{1}{2}\sqrt{(n+1)(n+2)}\,\vert X_N(p^0)\vert, &\\
&\text{\em (c)}& & \sqrt{\delta_{\Omega}(z)}\, ds_{\tilde b, \Omega}\big(z,X_{H}(z)\big)\to \sqrt{\frac{1}{2}(n+1)(n+2)\mathcal{L}_{\partial \Omega}\big(p,X_H(p^0)\big)},&\\
&\text{\em (d)}& &\beta_{\tilde b,\Omega}(z) \to (n+1)^n(n+2)^n\frac{\pi^n}{n!},&\\
&\text{\em (e)}& &R_{\tilde b, \Omega}(z,X)\to -\frac{2}{(n+1)(n+2)},\quad X\in\mathbb{C}^n\setminus\{0\}, &
\\
&\text{\em (f)}& &\Ric_{\tilde b,\Omega}(z,X)\to -\frac{1}{(n+2)},\quad X\in\mathbb{C}^n\setminus\{0\}. &
\end{align*}
 
\end{theorem}
Parts (b) and (c) of Theorem~\ref{bdy-n} can be viewed as analogues of Graham's result \cite{Gr} for the Kobayashi and Carath\' {e}odory metrics. Furthermore, using \cite[Theorem~1.17]{GK} and Theorem~\ref{bdy-n}~(e) implies that if $\Omega_1,\Omega_2 \subset \mathbb{C}^n$ are $C^2$-smoothly bounded strictly pseudoconvex domains equipped with the metrics $ds^2_{\tilde b,\Omega_1}$ and $ds^2_{\tilde b, \Omega_2}$, then every isometry 
\[F: \Big(\Omega_1, ds^2_{\tilde b,\Omega_1}\Big) \to \Big(\Omega_2, ds^2_{\tilde b, \Omega_2}\Big)\]
is either holomorphic or conjugate holomorphic. Moreover, for $n=1$, the holomorphic sectional curvature and the Ricci curvature coincide with the Gaussian curvature, which does not depend on $X$, and hence from Theorem~\ref{bdy-n}~(e), we observe that as $z\to \pa\Om$, 
\[R_{\tilde b, \Omega}(z,X)\to -\frac{1}{3}.\]
This particular case aligns with Theorem 1.4 (IV) of \cite{bk22}.
\begin{theorem}\label{loc}
Let $\Omega \subset \mathbb{C}^n$ be a $C^2$-smoothly bounded strictly pseudoconvex domain, and let $p^0\in \partial \Omega$. Then, for a sufficiently small neighbourhood $U$ of $p^0$, we have
\begin{itemize}
\item[(a)] $\lim_{z \to p^0} \dfrac{g_{\tilde b,U\cap \Omega}(z)}{g_{\tilde b,\Omega}(z)}=1,$
\item[(b)] $\lim_{z \to p^0} \dfrac{\beta_{\tilde b,U\cap \Omega}(z)}{\beta_{\tilde b,\Omega}(z)}=1$,
\item[(c)]  $\lim_{z \to p^0}\dfrac{ds_{\tilde b,U \cap \Omega}(z,X)}{ds_{\tilde b,\Omega}(z,X)}=1,$ 
\item[(d)] $\lim_{z \to p^0} \dfrac{2-R_{\tilde b,U \cap \Omega}(z,X)}{2-R_{\tilde b,\Omega}(z,X)}=1$,
\item[(e)] $ \lim_{z \to p^0} \dfrac{n+1-\Ric_{\tilde b, U \cap \Omega}(z,X)}{n+1-\Ric_{\tilde b,\Omega}(z,X)}=1$,
\end{itemize}
uniformly on $\{ \Vert X \Vert = 1\}$.
\end{theorem}

\noindent\textbf{Standing Assumptions:} Let \( z = ('z, z_n) \in \mathbb{C}^n \), where \( 'z = (z_1, \ldots, z_{n-1}) \). Throughout this article, we assume that \( \Omega  \) is a $C^2$-smoothly bounded strictly pseudoconvex domain. For any linear map \( L: \mathbb{C}^n \to \mathbb{C}^n \), its matrix representation is denoted by \( \mathsf{L} \), and \( \mathsf{I}_n \), the identity matrix, represents the identity map $i_n:\mathbb C^n\to \mathbb C^n$. Finally, let $\mathbb B^n(z, R)$ denote the ball centred at $z\in\mathbb C^n$ with radius $R>0$.
\medskip

\noindent \textbf{Acknowledgments:} The author would like to thank D. Borah for his constant encouragement during the project.

\section{The Ramadanov type theorem}
 The proof of our results is based on the stability of the Bergman kernel under Pinchuk's scaling sequence. We begin by recalling the change of coordinates associated with this. 

\subsection{Change of coordinates}
 Let $p^0\in \partial \Omega$ be a boundary point. There exists a $C^2$-smooth local defining function $r: U\to\mathbb R$ for $\Omega$, defined in a neighbourhood $U$ of a boundary point $p^0\in \partial \Omega$. We set \( \nabla_z r = \left(\partial r/\partial z_1, \ldots,\partial r/\partial z_n\right) \) and  \( \nabla_{\overline{z}} r = \overline{\nabla_z r} \). Therefore, the gradient of $r$ can be expressed as \( \nabla r = 2 \nabla_{\overline{z}} r \). Without loss of generality, let
\begin{equation}\label{normal-Re-z_n}
\nabla_{\overline z} r(p^0)=({'0},1) \text{ and } \frac{\partial r}{\partial z_n}(z)\neq 0 \text{ for all } z\in U.
\end{equation}
The following Lemma from \cite{Pin} provides the holomorphic change of coordinates near strictly pseudoconvex boundary points.

\begin{lemma}\label{change-coordinates}
There exist a family of automorphisms $\Psi_{p}:\mathbb{C}^n\to \mathbb{C}^n$ that depend continuously on $p \in \partial \Omega \cap U$, which satisfy the following:
\begin{itemize}
\item[(a)] The mapping $\Psi_{p}$ satisfies $\Psi_{p^0}=i_n$ and $\Psi_{p}(p)=0$.
\item[(b)] The local defining function $r_{p}=r\circ \Psi_{p}^{-1}$ of $\Omega_{p}=\Psi_{p}(\Omega)$ near the origin has the form 
\[
r_{p}(z)=2\text{Re}\big(z_n+G_{p}(z)\big)+L_{p}(z)+o(|z^2|),
\]
 where
\[
G_{p}(z)=\sum\limits_{\mu,\nu=1}^n a_{\mu \nu}(p)z_{\mu}z_{\nu}\quad \text{and} \quad L_{p}(z)=\sum\limits_{\mu,\nu=1}^n a_{\mu \overline{\nu}}(p)z_{\mu}\overline{z}_{\nu}.
\]
Here, the functions $a_{\mu \nu}(p), a_{\mu \overline \nu}(p)$ depend continuously on $p$ with $G_{p}('z,0)\equiv 0$ and $L_{p}('z,0)\equiv |'z|^2$.
\item[(c)] The mapping $\Psi_{p}$ maps the real normal $\eta_{p}=\{z=p+2t \nabla_{\overline z}r(p): t \in \mathbb{R}\}$ to $\partial \Omega$ at $p$ into the real normal $\{'z=y_n=0\}$ to $\partial \Omega_{p}$ at the origin.
\end{itemize}
\end{lemma}

The definition of $\Psi_p$ and its derivative plays an important role in examining the boundary asymptotics of the invariant objects, so we briefly recall its construction. For each $p\in\partial\Omega\cap U$, $\Psi_{p}(z)=\Phi_{3 }^{p}\circ \Phi_{2}^{p}\circ \Phi_{1 }^{p}(z)$, where each $\Phi_{i}^{p}:\mathbb{C}^n \to \mathbb{C}^n$ is an automorphism. We fix $p\in\partial\Omega\cap U$. Recall that the Taylor expansion of the defining function $r$ near $p$ has the form 
\begin{multline}\label{rho-0}
r(z)=2\Re\left(\sum_{\mu=1}^n \frac{\partial r(p)}{\partial z_{\mu}}(z_{\mu}-p_{\mu}) + \frac{1}{2} \sum_{\mu,\nu=1}^n \frac{\partial^2 r(p)}{\partial z_{\mu} \partial z_{\nu}} (z_{\mu}-p_{\mu})(z_{\nu}-p_{\nu}) \right)\\+\sum_{\mu,\nu=1}^n \frac{\partial^2 r(p)}{\partial z_{\mu}\partial \overline z_{\nu}} (z_{\mu}-p_{\mu})(\overline z_{\nu}-\overline p_{\nu}) +o(\vert z-p\vert^2).
\end{multline}

  The map $w=\Phi_{1}^{p}(z)$ is an affine transformation defined as $\Phi_{1}^{p}(z)=P_p(z-p)$, where  $P_{p}:\mathbb C^n\to \mathbb C^n$ is the linear map whose matrix is
\begin{align}\label{P-matrix}
\mathsf{P}_{p}=\begin{pmatrix}
\frac{\partial r(p)}{\partial \overline z_n} & 0 & \cdots & 0 & -\frac{\partial r(p)}{\partial \overline z_1}\\
0 & \frac{\partial r(p)}{\partial \overline z_n} & \cdots & 0 & -\frac{\partial r(p)}{\partial \overline z_2}\\
\vdots & \vdots & \cdots &\vdots & \vdots\\
0 & 0 & \cdots & \frac{\partial r(p)}{\partial \overline z_n} & -\frac{\partial r(p)}{\partial \overline z_{n-1}}\\
\frac{\partial r(p)}{\partial z_1} & \frac{\partial r(p)}{\partial z_2} & \cdots & \frac{\partial r(p)}{\partial z_{n-1}} & \frac{\partial r(p)}{\partial z_n}
\end{pmatrix},
\end{align}
i.e.,
    \begin{equation}\label{phi_1-defn}
\begin{aligned}
w_\mu & =\frac{\partial r(p)}{\partial \overline z_n}(z_\mu-p_\mu)-\frac{\partial r(p)}{\partial \overline z_\mu}(z_n-p_n)\quad \text{for}\quad 1\leq \mu\leq n-1,\\
w_n & =\sum\limits_{\mu=1}^n \frac{\partial r(p)}{\partial z_{\mu}}(z_{\mu}-p_{\mu}).
\end{aligned}
\end{equation}
By \eqref{normal-Re-z_n}, $\Phi_{1}^{p}$ is nonsingular and $\Phi_{1}^{p}(p)=0$. Also,
\begin{equation}\label{phi1-n_z}
\Phi_{1}^{p} \big(p+t\nabla_{\overline z}r(p)\big) =t\mathsf{P}  _p\nabla_{\overline z}r(p)=\big({'0}, t\vert \nabla_{\overline z}r(p)\vert^2\big).
\end{equation}
Thus, $\Phi_{1}^{p}$ maps $\eta_p$ to $\Re w_n$, which is the real normal to the boundary of $\Omega_{p}^{1}=\Phi_{1}^{p}(\Omega)$. This can be seen from the Taylor expansion of the defining function  $r_{1}^{p}=r\circ( \Phi_{1 }^{p})^{-1}$ of $\Omega_{p}^1$ near the origin. Indeed, by expressing (\ref{rho-0}) in $w$ coordinates and then replacing $w$ by $z$, we have

\begin{equation}\label{r-1}
r_{1}^p(z)=2\text{Re}\left(z_n+G_p^1(z)\right)+L_{p}^1(z)+o(\vert z \vert^2),
\end{equation}
where
\begin{equation}\label{a1-b1}
\begin{aligned}
G_p^1(z) &=\sum_{\mu,\nu=1}^{n} a^1_{\mu \nu}(p) z_{\mu} z_{\nu}, \quad \begin{pmatrix}a^1_{\mu \nu}(p)\end{pmatrix}=\frac{1}{2}(\mathsf{P}_{p}^{-1})^t \begin{pmatrix} \frac{\partial^2 r(p)}{\partial z_{\mu} \partial z_{\nu}} \end{pmatrix}\mathsf{P}_{p}^{-1},\\
L_p^{1}(z) & =\sum_{\mu,\nu=1}^n b^1_{\mu \overline \nu}(p) z_{\mu} \overline z_{\nu}, \quad 
\begin{pmatrix}b^1_{\mu\overline \nu}(p)\end{pmatrix}=(\mathsf{P}_{p}^{-1})^{*}\begin{pmatrix} \frac{\partial^2 r(p)}{\partial z_{\mu}\partial \overline z_{\nu}} \end{pmatrix}\mathsf{P}_{p}^{-1}.
\end{aligned}
\end{equation}
From (\ref{phi_1-defn}), we have as $p\to p^0,\;\Phi_{1}^{p}(z)\to\Phi_{1}^{p^0}(z)$ uniformly on compact subsets of $\mathbb C^n$. Furthermore, $(\Phi_{1}^{p})'(z)\to(\Phi_{1}^{p^0})'(z)$ in the operator norm, and uniformly in $z\in\mathbb{C}^n$, since $(\Phi_{1}^{p})'(z)$ is identically equal to $\mathsf P_p$.

Next, the map $w=\Phi_{2}^{p}(z)$ is a polynomial automorphism, defined as 
\begin{equation}\label{phi_2-defn}
w=\left('z, z_n+\sum\limits_{\mu,\nu=1}^{n-1}a_{\mu\nu}^1(p)z_{\mu}z_{\nu}\right).
\end{equation}
Clearly, $\Phi_{2}^{p}$ fixes the points in $\Re z_n$-axis. Again, we relabel the new coordinates $w$ by $z$, and obtain the local defining function $ r \circ (\Phi_{1}^{p})^{-1} \circ (\Phi_{2}^{p})^{-1} $ of the domain $\Omega_{p}^2= \Phi_{2}^{p}\circ \Phi_{1}^{p}(\Omega)$ near the origin has the form \eqref{r-1} with $
a_{\mu\nu}^1(p)=0$ for $1\leq \mu, \nu \leq n-1$. From (\ref{a1-b1}), $a_{\mu\nu}^1(p)$ depends continuously on $p$, which yields as \( p \to p^0 \), \( \Phi_{2}^{p}(z) \to \Phi_{2}^{p^0}(z) \) uniformly on compact subsets of \( \mathbb{C}^n \). Furthermore,
\begin{equation}\label{der-phi2}
(\Phi_2^{p})'(z)= \begin{pmatrix}\mathsf{I}_{n-1} & 0\\
\begin{pmatrix}\displaystyle \sum_{\mu=1}^{n-1} a^1_{\mu \ga}(p) z_{\mu} + \displaystyle \sum_{\nu=1}^{n-1} a^1_{\ga\nu}(p) z_{\nu}\end{pmatrix}_{\ga=1, \ldots, n-1}& 1 \end{pmatrix}.
\end{equation}
Therefore, as \( p \to p^0 \), \( (\Phi_{2}^{p})'(z) \to (\Phi_{2}^{p^0})'(z) \) in operator norm as well as uniformly on compact subsets of \( \mathbb{C}^n \).

Lastly, the map $\Phi_{3}^{p}$ is chosen so that the Hermitian form $L_{p}^1(z)$ satisfies $L_{p}^1('z,0)=\vert 'z\vert^2$. Since $\Omega$ is strictly pseudoconvex, and in the current coordinates, the complex tangent space to $\partial \Omega$ at $p$ is $H_{p}(\partial\Omega)=\{z_n=0\}$. We obtain, the form $L_{p}^1('z,0)$ is strictly positive definite. Therefore, there exists a unitary map $U_{p}:\mathbb C^{n-1}\rightarrow \mathbb C^{n-1}$ such that $L_{p}^1(U_{p}('z),0)=\sum_{i=1}^n\la_i(p)|z_i|^2$, where $\la_{i}(p)>0$ denote the eigenvalues of $L_{p}^1('z,0)$. Then, using the stretching map $R_{p}=\Big(z_1/\sqrt {\la_1(p)},\ldots,z_{n-1}/\sqrt{\la_{n-1}(p)}\Big),$
we define a linear map  \[A_{p}=R_{p}\circ U_{p},\]
that satisfies $L_{p}^1(A_{p}('z),0)=|'z|^2$. Observe that both $U_{p}$ and $R_{p}$, and hence $A_{p}$, can be chosen to depend continuously on $p$, with its derivative at any point also varying continuously with $p$. Let $w=\Phi_{3}^{p}(z)$ be defined as
\begin{equation}\label{phi_3-defn}
    w=(A_{p}('z),z_n),
\end{equation}
and relabel $w$ with $z$. Then, the local defining function $r_{p}=r\circ \Psi_{p}^{-1}$ of the domain $\Omega_{p}=\Psi_{p}(\Omega)=\Phi_3^p\circ\Phi_2^p\circ \Phi_1^p (\Omega)$ near the origin has the Taylor series expansion as described in (b). From the definition, observe that
$\Phi_3^p$ fixes the points in the $\Re z_n$-axis. The linear map $\Phi_3^p$ can be chosen so that both the map and its derivative vary continuously with $p$. Therefore, it follows from the construction of the maps $\Phi_{i}^{p},~i=1,2,3$ that $\Psi_{p}$ satisfy (a) and (c) and depend continuously on $p$.

\subsection{Scaling of $\Omega$}
Using the strict pseudoconvexity of $\partial\Omega$, there exist local holomorphic coordinates $z_1,\ldots, z_n$ on a sufficiently small neighbourhood $U$ such that $p^0=0\in\partial\Omega$, and 
\begin{equation}\label{normal form}
r(z)=2\Re z_n+|'z|^2+o\big(\Im z_n,|'z|^2\big), \quad z \in U,
\end{equation}
with a constant $0<c_0<1$, so that
\begin{align}\label{subset1}
U \cap \Omega\subset D:=\big\{z \in \mathbb C^n: 2\Re z_n+c_0|'z|^2<0\big\}.
\end{align}
As specified, we will work in the above coordinates, with $U,c_0$ and $p^0=0$.

Without loss of generality, let $(\zeta^j)_{j\geq 1}$ be a sequence in $U\cap\Omega$ converging to $0$, and for each $j$, there exists a unique $p^j\in \partial \Omega \cap U$ such that $|\zeta^j-p^j|= \delta_{\Omega}(\zeta^j)=\delta_j$. Observe that as $j\to \infty$, $p^j\to 0$ and $\delta_j\to 0$. Let $\Phi_{i}^{p^j}=\Phi_{i}^{j},\;P_j=P_{p^j},\;A_j=A_{p^j},\;\Psi_j=\Psi_{p^j},\;\Omega_j=\Omega_{p^j}$, $r_j=r_{p^j}$, $G_j=G_{p^j}$, and $L_j=L_{p^j}$. By Lemma \ref{change-coordinates}, near $0$, we have
\begin{align*}
r_j(z)= 2\text{Re}\,\big(z_n+G_j(z)\big)+L_j(z)+o\big(|z|^2\big).
\end{align*}
Furthermore, due to the strict pseudoconvexity of 
$\partial \Omega$ near $0$, shrinking $U$ if necessary and reducing $c_0$ in \eqref{subset1} if necessary, we obtain
\begin{align}\label{subset2}
\Psi_j(U\cap \Omega) \subset D
\end{align}
 for all $j$ large. Now, set $ \Omega_j'=\Psi_j(U\cap \Omega)$, $q^j=\Psi_j(\zeta^j)$ and  for all $j$ large, $\eta_j=\delta_{\Omega_j'}(q^j)=\delta_{\Omega_j}(q^j)$. From (\ref{phi1-n_z}) and the fact that $\Phi_2^j,~j=2,3$ fixes the points in $\Re z_n$-axis, it follows that
\begin{equation}\label{q-j}
   q^j=('0, -\delta_j|\nabla_{\overline z}r(p^j)|).
\end{equation}
where $\zeta^j=p^j-\delta_j\nabla_{\overline z}r(p^j)/|\nabla_{\overline z}r(p^j)|$.
Consequently, $\eta_j=\delta_j|\nabla_{\overline z}r(p^j)|,$ and hence by (\ref{normal-Re-z_n}),
\begin{equation}\label{eta-j/de-j}
    \frac{\eta_j}{\delta_j}=|\nabla_{\overline z}r(p^j)|\to |\nabla_{\overline z}r(0)|=1.
\end{equation}
Also, set $\mathsf Q_j=\Psi_j'(\zeta^j)$ and as $j\to\infty$, we have
\begin{equation}\label{conv-Q-j}
    \mathsf Q_j\to \mathsf I_{n}
\end{equation}
in the operator norm, since \begin{align*}
(\Phi_2^j)'\big(\Phi_1^j(\zeta^j)\big)=\mathsf{I}_n,
\end{align*}
and \[
(\Phi^j_3)'=\begin{bmatrix} \mathsf A_j & 0\\ 0 &1\end{bmatrix},
\] where $\mathsf A_j$ converges to $\mathsf{I}_{n-1}$ in the operator norm.

 Next, consider the anisotropic dilation map $T_j : \mathbb{C}^n \to \mathbb{C}^n$ defined by
\begin{equation}\label{defn-T_j}
T_j(z)=\left(\frac{z_1}{\sqrt{\eta_j}},\ldots,\frac{z_{n-1}}{\sqrt{\eta_j}},\frac{z_n}{\eta_j}\right),
\end{equation}
Note that 
\begin{equation}\label{det-T-matrix}
    \det\mathsf  T_j=\eta_j^{\frac{-(n+1)}{2}}.
\end{equation}
Let $\widetilde \Omega_j'=T_j(\Omega_j')$ and $\widetilde{\Omega}_j=T_j(\Omega_j)$. We call $S_j=T_j\circ \Psi_j$, the scaling maps. Using (\ref{q-j}) in (\ref{defn-T_j}), 
\[S_j(\zeta^j)=T_j(q^j)=('0,-1)=b^*\in \widetilde \Omega_j'\subset\widetilde \Omega_j.\]
The defining function for $\widetilde \Omega_j'$ and $\widetilde \Omega_j$ near $0$, is given by
\[
\tilde r_j(z)=\frac{1}{\eta_j} r_j\Big(T_j^{-1}(z)\Big)=2\text{Re}\Bigg(z_n+\frac{1}{\eta_j}G_j\Big(T_j^{-1}(z)\Big)\Bigg)+\frac{1}{\eta_j}L_j\Big(T_j^{-1}(z)\Big)+o\Big(\eta_j^{1/2}|z|^2\Big).
\]
As $G_j(z)$ and $L_j(z)$ satisfy Lemma~\ref{change-coordinates} (b), it follows that
\[
\lim_{j\to \infty}\frac{1}{\eta_j}G_j(T_j^{-1}z)=0\quad \text{and}\quad \lim_{j\to \infty}\frac{1}{\eta_j}L_j(T_j^{-1}z)=|'z|^2,
\]
in $C^2$-topology on compact subsets of $\mathbb{C}^n$. Also, $o(\eta_j^{1/2}\vert z\vert ^2)\to 0$ as $j \to \infty$ in $C^2$-topology on compact subsets of $\mathbb C^n$. Hence, the sequence of defining functions $\tilde r_j(z)$ converge to 
\[
r_{\infty}(z)=2\text{Re}\,z_n+|'z|^2
\]
 in $C^2$-topology on compact subsets of $\mathbb C^n$. Therefore, the sequence of domains $\widetilde \Omega_j'$ and $\widetilde \Omega_j$ converge in the local Hausdorff sense to the Siegel upper half-space
\[
\Omega_{\infty}=\big\{z\in \mathbb C^n:2\text{Re}\,z_n+|'z|^2<0\big\}.
\]
Moreover, by (\ref{subset2}), it can be seen that for each $\epsilon>0$, 
\begin{equation}\label{subset3}
    H(\widetilde \Omega_j')\subset (1+\epsilon)H(\Omega_\infty),
\end{equation}
for all $j$ large, where $H(z)=\left(\frac{\sqrt{2}\,{'z}}{z_n-1}, \frac{z_n+1}{z_n-1}\right)$ is a biholomorphism of the domain 
\[
\Omega_H=\mathbb{C}^n \setminus \{z: z_n=1\},
\]
onto itself with $H(\Omega_\infty)=\mathbb{B}^n$ and $H^{-1}=H$. Furthermore,
\begin{equation}\label{der-Phi}
H(b^*)=0,~ H'(b^*)=-\diag\{1/\sqrt 2,\ldots,1/\sqrt 2, 1/2\},~\det H'(b^*)=(-1)^n2^{\frac{-(n+1)}{2}}.
\end{equation}
We refer to the proof of \cite[Lemma 5.3]{BV-ns} for further details of (\ref{subset3}). In this context, the stability result for $\widetilde \Omega_j'$ follows from a Ramadanov-type result stated in \cite[Lemma 2.1]{BBMV1}.
\begin{lemma}\label{local-stab}
    The sequence of Bergman kernels $K_{\widetilde \Omega_j'}(z, w)$ converges to $K_{\Omega_\infty}(z, w)$ uniformly on compact subsets of $\Omega_\infty\times \Omega_\infty$, along with all the derivatives. 
\end{lemma}
However, condition (\ref{subset3}) may not hold for $\widetilde \Omega_j$, since (\ref{subset2}) is not necessarily satisfied. Despite this, we have 
\begin{lemma}\label{ram}
    The sequence of Bergman kernels $K_{\widetilde\Omega_j}(z, w)$ converges to $K_{\Omega_\infty}(z, w)$ uniformly on compact subsets of $\Omega_\infty\times \Omega_\infty$, along with all the derivatives. 
\end{lemma}
To prove this, we require the following Lemma.
\begin{lemma}\label{phi-ij-inverse}
 For each $i=1,2,3$, the sequence $(\Phi_{i}^{j})^{-1}$ converges to $(\Phi_{i})^{-1}$ uniformly on compact subsets of $\mathbb{C}^n$, where 
    \[\Phi_1^{-1}(w)=i_n(w),~\Phi_2^{-1}(w)=\left( 'w, w_n-\sum\limits_{\mu,\nu=1}^{n-1}a_{\mu\nu}^1(0)w_{\mu}w_{\nu}\right),~\Phi_3^{-1}(w)=(i_{n-1}('w),w_n).\]
\end{lemma}
\begin{proof}
    From (\ref{phi_1-defn}), (\ref{phi_2-defn}), and (\ref{phi_3-defn}), we have
    \begin{itemize}
        \item $(\Phi_{1}^{j})^{-1}(w)=P_{j}^{-1}(w)+p^j,$
        \item $(\Phi_{2}^{j})^{-1}(w)=\left( 'w, w_n-\sum\limits_{\mu,\nu=1}^{n-1}a_{\mu\nu}^1(p^j)w_{\mu}w_{\nu}\right),$ and
        \item $(\Phi_{3}^{j})^{-1}(z)=\big(A_{j}^{-1}('w),w_n\big)$.
    \end{itemize}
    By the construction of \(\Phi_{i}^{j}\) for \(i=1,2,3\), it follows that the matrices \(\mathsf{P}_{j}\) and \(\mathsf{A}_{j}\) are invertible, with \(\mathsf{P}_{j}\to\mathsf{I}_{n}\) and  \(\mathsf{A}_{j}\to \mathsf{I}_{n-1}\) as \(p^j\to 0\in\partial\Omega\). Furthermore, \(a_{\mu\nu}(p^j) \to a_{\mu\nu}(0)\). This completes the proof.  
\end{proof}
\begin{proof}[Proof of Lemma~\ref{ram}]
   Let $z\in \Omega_\infty$. For sufficiently small $c_1>0$, $\mathbb B^n(z, c_1)\Subset\widetilde\Omega_\infty$. It follows that for all $j$ large,
 $\mathbb B^n(z, c_1)\Subset \widetilde{\Omega}_j'\subset \widetilde\Omega_j$. Consequently, for such $j$'s, 
    \begin{equation*}
        K_{\widetilde\Omega_j}(z)\leq K_{\mathbb B^n(z, c_1)}(z).
    \end{equation*}
   Therefore, for any $z, w\in \mathbb B^n(z, c_1/2)$,
    \begin{align*}
        \Big|K_{\widetilde\Omega_j}(z, w)\Big|\leq \sqrt{K_{\widetilde\Omega_j}(z)}\sqrt{ K_{\widetilde\Omega_j}(w)}\leq \max_{z\in \overline{ \mathbb{B}^n(z, c_1/2)}}{K_{\mathbb B^n(z, c_1)}(z)}
    \end{align*}
     for all $j$ large. Thus, $\Big(K_{\widetilde\Omega_j}(z, w)\Big)$ is a locally uniformly bounded sequence on $\Omega_\infty\times \Omega_\infty$. By Montel's theorem, there exists a subsequence converging locally uniformly to a function, say $K_{\infty}(z, w)$ on $\Omega_\infty\times \Omega_\infty$, which is holomorphic in $z$ and anti-holomorphic in $w$.
     
     \noindent We claim that
     \[K_{\infty}\equiv K_{\Omega_\infty}.\]
  To prove this claim, it is sufficient to show
  \begin{equation}\label{diag}
      K_{\infty}(z, z)\equiv K_{\Omega_\infty}(z,z),
  \end{equation}
 since the difference $K_{\infty}(z, \overline{w})-K_{\Omega_\infty}(z, \overline{w})$ is holomorphic in $ \mathbb B^n(b^*,c_2)\times \mathbb B^n(b^*,c_2)^*$. Here, $c_2>0$ is sufficiently small such that $\mathbb B^n(b^*,c_2)^*=\{\overline z: z\in \mathbb B^n(b^*,c_2)\}\Subset\Omega_\infty$. The above arguments also show that any convergent subsequence of $K_{\widetilde \Omega_j}(z,w)$ have to converge to $K_{\Omega_\infty}(z,w)$, and hence $K_{\widetilde\Omega_j}(z,w)$ itself converges to $K_{\Omega_\infty}(z,w)$. Moreover, the convergence of derivatives follows from the harmonicity of the functions $K_{\widetilde\Omega_j}(z,w)$.
 
We now prove (\ref{diag}). Let $z\in\Omega_\infty$. For all $j$ large, we have $z\in\widetilde\Omega_j'\subset\widetilde\Omega_j$. Then, using the transformation rule for the Bergman kernel,
\begin{align*}\label{uniq-of-limit}
       \lim_{j\to \infty}K_{\widetilde{\Omega}_j}(z)=\lim_{j\to \infty}\frac{K_{\Omega}\Big(S_j^{-1}(z)\Big)}{K_{U\cap\Omega}\Big(S_j^{-1}(z)\Big)}K_{\widetilde\Omega_j'}(z).
     \end{align*}
By Lemma \ref{phi-ij-inverse}, 
\[ S_j^{-1}(z)\to 0\in\partial\Omega.\]
Therefore, the localization of the Bergman kernel concludes our claim.
\end{proof}

\section{Boundary asymptotics}
In this section, we establish Theorems \ref{loc} and \ref{bdy-n}. To begin, we note the following. Recall that $S_j=T_j \circ \Psi_j$, $S_j(U\cap \Omega)=\widetilde \Omega_j'$, $S_j(\Omega)=\widetilde \Omega_j$, $S_j(\zeta^j)=b^{*}=('0,-1)$,  and 
\begin{equation}\label{der-S_j}
S_j'(\zeta^j)X=
\left(\frac{'(\mathsf{Q}_jX)}{\sqrt{\eta_j}}, \frac{(\mathsf{Q}_jX)_n}{\eta_j}\right),
\end{equation}
where $\mathsf Q_j=\Psi_j'(\zeta^j)$, and $\zeta^j=p^j-\delta_j\nabla_{\overline z}r(p^j)/|\nabla_{\overline z}r(p^j)|$. In what follows, we will repeatedly make use of (\ref{eta-j/de-j}), (\ref{conv-Q-j}), (\ref{det-T-matrix}), (\ref{der-Phi}), and (\ref{der-S_j}), without explicitly referring to them each time. We also need the following Lemma.
\begin{lemma}\label{stability-Fuks}
For $z \in \Omega_\infty,$ and $X\in \mathbb{C}^n \setminus \{0\}$, the following limits hold as $j\to\infty$,
\begin{align*}
 g_{\tilde b, \widetilde \Omega_j}(z)\to g_{\tilde b, \Omega_\infty}(z), \quad \beta_{\tilde b, \widetilde \Omega_j}(z)\to \beta_{\tilde b, \Omega_\infty}(z),
\end{align*}
and also
\begin{align*}
& ds_{\tilde b, \widetilde \Omega_j}(z, X)\to ds_{\tilde b, \Omega_{\infty}}(z,X),  \quad  R_{\tilde b, \widetilde \Omega_j}(z, X)\to R_{\tilde b, \Omega_{\infty}}(z,X),\quad  \Ric_{\tilde b, \widetilde \Omega_j}(z, X)\to \Ric_{\tilde b, \Omega_{\infty}}(z,X).
\end{align*}
 Furthermore, the first and second convergences are uniform on compact subsets of \,$\Omega_{\infty}$ and the third, fourth and fifth convergences are uniform on compact subsets of\, $\Omega_{\infty}\times \mathbb C^n$.
\end{lemma}

\begin{proof}
The Kobayashi--Fuks metric on a bounded domain 
$\Omega$ has a K\"ahler potential $\log(K^{n+1}_{\Omega}g_{b,\Omega})$, i.e.,
\[
g^{\tilde b,\Omega}_{\alpha \overline \beta} = \frac{\partial^2 \log(K^{n+1}_{\Omega}g_{b,\Omega})}{\partial z_{\alpha}\partial \overline z_{\beta}}.
\]
To complete the proof, it suffices to show that
\[
K^{n+1}_{\widetilde\Omega_j}g_{b,\widetilde\Omega_j} \to K^{n+1}_{\Omega_{\infty}}g_{b,\Omega_\infty}
\]
uniformly on compact subsets of $\Omega_{\infty}$, along with all derivatives. This is an immediate consequence of Lemma \ref{ram}. 
\end{proof}
\begin{remark}\label{rem1}
    Lemma \ref{stability-Fuks} also applies to \(\widetilde{\Omega}_j'\), which can be seen by using Lemma \ref{local-stab}.
\end{remark}

\subsection{Proof of Theorem~\ref{bdy-n}} Before proving this theorem, we establish the following Lemma.
\begin{lemma}\label{comps-on-ball}
For the unit ball $\mathbb{B}^n \subset \mathbb{C}^n$, we have
\begin{itemize}
\item [(a)] $g_{\tilde b, \mathbb{B}^n}(z) =\dfrac{(n+1)^n(n+2)^n}{(1-\vert z\vert^2)^{n+1}}$,
\item[(b)] $\beta_{\tilde b, \mathbb{B}^n}(z) =(n+1)^{n}(n+2)^{n} \dfrac{\pi^n}{n!}$,
\item[(c)] $ R_{\tilde b, \mathbb{B}^n}(z,X) =-\dfrac{2}{(n+1)(n+2)},\quad X\in\mathbb{C}^n\setminus\{0\},$
\item[(d)] $\Ric_{\tilde b, \mathbb{B}^n}(z,X) =-\dfrac{1}{(n+2)},\quad X\in\mathbb{C}^n\setminus\{0\}.$
\end{itemize}
\end{lemma}
\begin{proof}
From Example \ref{ball},
\[
g^{\tilde b, \mathbb{B}^n}_{\alpha \overline \beta}(z)=(n+2)g^{b, \mathbb{B}^n}_{\alpha\overline \beta}(z).
\]
This implies \begin{align*}
g_{\tilde b,\mathbb{B}^n}&=(n+2)^ng_{b,\mathbb{B}^n},\\
   \beta_{\tilde b,\mathbb{B}^n}&=(n+2)^n\beta_{b,\mathbb{B}^n},\\
  R_{\tilde b, \mathbb{B}^n}(z,X)&=\frac{1}{n+2} R_{b, \mathbb{B}^n}(z,X),\\
\Ric_{\tilde b, \mathbb{B}^n}(z,X)&=\frac{1}{n+2} \Ric_{b, \mathbb{B}^n}(z,X).
\end{align*}
Recall
\[ K_{\mathbb{B}^n}(z)=\frac{n!}{\pi^n}\frac{1}{(1-\vert z \vert^2)^{n+1}},\]  
and
\[g^{b, \mathbb{B}^n}_{\alpha\overline \beta}(z)=(n+1)\sum_{\alpha,\beta=1}^n \left(\frac{\delta_{\alpha\overline\beta}}{1-\vert z \vert^2}+\frac{\overline z_{\alpha}z_{\beta}}{(1-\vert z \vert^2)^2}\right) dz_{\alpha} d\overline z_{\beta}.\]
This implies that
\[
g_{b,\mathbb{B}^n}(z)=\frac{(n+1)^n}{(1-\vert z \vert^2)^{n}}\Bigg(1+\frac{|z|^2}{(1-|z|^2)}\Bigg)=\frac{(n+1)^n}{(1-\vert z \vert^2)^{n+1}},
\]
and hence
\[\beta_{ b, \mathbb{B}^n}(z)=\frac{g_{b,\mathbb{B}^n}(z)}{K_{\mathbb{B}^n}(z)}=(n+1)^{n} \frac{\pi^n}{n!}.\]
Thus, the proof of (a) and (b) follows. Then, by performing a straightforward calculation for the Bergman metric, we have
\[R_{b,\mathbb{B}^n}(z, X)=-\frac{2}{n+1}\text{ and }\Ric_{b,\mathbb{B}^n}(z, X)=-1,\]
for any $X\in\mathbb{C}^n\setminus\{0\}$, which completes the proof of (c) and (d).
\end{proof}

We are now ready for the proof of Theorem~\ref{bdy-n}.
\begin{proof}[Proof of Theorem~\ref{bdy-n}]
Using Lemma \ref{stability-Fuks}, the proofs of (a), (b), and (c) follow directly from the reasoning outlined in \cite[Theorem 1.2]{bk22} for $\widetilde\Omega_j'$. Hence, it remains to establish (d), (e), and (f). To proceed, we again utilize Lemma \ref{stability-Fuks} and Lemma \ref{comps-on-ball}, along with the transformation rules for the Kobayashi–Fuks metric and the Bergman kernel.
\begin{itemize}
    \item [(d)] We have 
\begin{equation*}
    \beta_{\tilde b, \Omega}(\zeta^j) = \beta_{\tilde b, \widetilde\Omega_j}(b^*)  \to \beta_{\tilde b,\Omega_{\infty}}(b^*) =\beta_{\tilde b,\mathbb{B}^n}(0)=(n+1)^n(n+2)^n\frac{\pi^n}{n!}.
\end{equation*}
    
\item[(e)]  For any $X\in\mathbb C^n\setminus\{0\}$, let $\hat{X}= \lim_{j \to \infty} \mathsf{T}_j\cdot\mathsf{Q}_j(X)/\vert \mathsf{T}_j\cdot\mathsf{Q}_j(X)\vert$. Then
\begin{multline*}
R_{\tilde b,\Omega}(\zeta^j,v) =R_{\tilde b,\widetilde\Omega_j}(b^*, \mathsf{T}_j\cdot\mathsf{Q}_j(X)) = R_{\tilde b,\widetilde\Omega_j}\left(b^*, \frac{\mathsf{T}_j\cdot\mathsf{Q}_j(X)}{\vert \mathsf{T}_j\cdot\mathsf{Q}_j(X)\vert}\right) \\ \to R_{\tilde b,\Omega_\infty}\big(b^*, \hat{X}\big)
=R_{\tilde b,\mathbb{B}^n}\big(0, \Phi'(b^*)\hat{X}\big) =-\frac{2}{(n+1)(n+2)}.
\end{multline*}
\item[(f)] Proceeding as in (e), we obtain
\[
\Ric_{\tilde b,\Omega}(\zeta^j,X) \to \Ric_{\tilde b,\mathbb{B}^n}\big(0, \Phi'(b^*)\hat{X}\big) = -\frac{1}{n+2}.
\]

\end{itemize}
Hence, we are done.
\end{proof}

\begin{proof}[Proof of Theorem~\ref{loc}]   By the transformation rules for the  Bergman kernel and the Kobayashi-Fuks metric, we have  
\[  
\frac{g_{\tilde b, U \cap \Omega}(\zeta^j)}{g_{\tilde b, \Omega}(\zeta^j)} = \frac{g_{\tilde b, \widetilde{\Omega}_j'}(b^*)}{g_{\tilde b, \widetilde{\Omega}_j}(b^*)} \quad \text{and} \quad \frac{\beta_{\tilde b, U \cap \Omega}(\zeta^j)}{\beta_{\tilde b, \Omega}(\zeta^j)} = \frac{\beta_{\tilde b, \widetilde{\Omega}_j'}(b^*)}{\beta_{\tilde b, \widetilde{\Omega}_j}(b^*)}.  
\]  
Using Lemma \ref{stability-Fuks} and the Remark \ref{rem1}, it follows that
\[\frac{g_{\tilde b, U \cap \Omega}(\zeta^j)}{g_{\tilde b, \Omega}(\zeta^j)} \to \frac{g_{\tilde b, \Omega_\infty}(b^*)}{g_{\tilde b, \Omega_\infty}(b^*)}=1\quad \text{and} \quad \frac{\beta_{\tilde b, U \cap \Omega}(\zeta^j)}{\beta_{\tilde b, \Omega}(\zeta^j)}\to \frac{\beta_{\tilde b, \Omega_\infty}(b^*)}{\beta_{\tilde b, \Omega_\infty}(b^*)}=1\]
as \(j \to \infty\). Next, by the invariance of the metric, we have
\begin{multline*}  
\frac{ds_{\tilde b, U \cap \Omega}(\zeta^j, X)}{ds_{\tilde b, \Omega}(\zeta^j, X)} = \frac{ds_{\tilde b, \widetilde{\Omega}_j'}\Big(b^*, S_j'(\zeta^j) X\Big)}{ds_{\tilde b, \widetilde{\Omega}_j}\Big(b^*, S_j'(\zeta^j) X\Big)}   
= \frac{\eta_j ds_{\tilde b, \widetilde{\Omega}_j'}\left(b^*, \left(\frac{'(\mathsf{Q}_j X)}{\sqrt{\eta_j}}, \frac{(\mathsf{Q}_j X)_n}{\eta_j}\right)\right)}{\eta_j ds_{\tilde b, \widetilde{\Omega}_j}\left(b^*, \left(\frac{'(\mathsf{Q}_j X)}{\sqrt{\eta_j}}, \frac{(\mathsf{Q}_j X)_n}{\eta_j}\right)\right)} \\  
= \frac{ds_{\tilde b, \widetilde{\Omega}_j'}\Big(b^*, \left(\sqrt{\eta_j} '(\mathsf{Q}_j X), (\mathsf{Q}_j X)_n\right)\Big)}{ds_{\tilde b, \widetilde{\Omega}_j}\Big(b^*, \left(\sqrt{\eta_j} '(\mathsf{Q}_j X), (\mathsf{Q}_j X)_n\right)\Big)}  
\to \frac{ds_{\tilde b, \Omega_\infty}\left(b^*, \left('0, X_n\right)\right)}{ds_{\tilde b, \Omega_\infty}\left(b^*, \left('0, X_n\right)\right)} = 1.  
\end{multline*}  

Moreover, the uniform convergence on \(\{\|X\| = 1\}\) follows directly from Lemma \ref{stability-Fuks} and  Remark \ref{rem1}. Similarly, the remaining assertions can be established. Hence, the proof follows. 
\end{proof}

\begin{remark}
The analogue version of Lemma \ref{ram} holds for the Narasimhan-Simha type metric of order $d\geq 0$, using the localization of the weighted Bergman kernel (cf. \cite[Theorem 3.4]{BV-ns})  Therefore, the study of the invariants associated with the Narasim

\noindent -han--Simha type metric of order $d$  can be simplified too.
\end{remark}


\end{document}